\documentclass[preprint,12pt]{elsarticle}

\usepackage[T5]{fontenc}
\usepackage[utf8]{inputenc}

\usepackage{amsfonts,amsmath,bm,dsfont,mathrsfs,amssymb,pifont}
\usepackage{amsthm} 
\usepackage{cases}
\usepackage{stmaryrd}

\usepackage{graphicx}
\usepackage{enumitem}
\usepackage{color,calc}

\usepackage{hyperref}
\hypersetup{
  colorlinks=true,
  allcolors=refblue,
  urlcolor=refblue,
}
\definecolor{refblue}{RGB}{26,13,171}

\def\cA{{\mathcal A}}

\def\cK{{\mathcal K}}

\def\cO{{\mathcal O}}

\def\cS{{\mathcal S}}

\def\cU{{\mathcal U}}

\newcommand{\trace}{{\sf trace}}
\renewcommand\Re{{\mathbb R}}

\newcommand{\dist}{{\sf dist}}

\newtheorem{theorem}{Theorem}[section]
\newtheorem{lemma}[theorem]{Lemma}
\newtheorem{proposition}[theorem]{Proposition}

\newtheorem{assumption}[theorem]{Assumption}

\begin{document}

\begin{frontmatter}

\title{On Error Bounds for Rank-Constrained Affine Matrix Sets}

\author[tsinghua,polyu]{Ruoning Chen}

\author[polyu]{Defeng Sun\fnref{funding_ds}}

\author[tsinghua]{Liping Zhang\corref{cor_author}\fnref{funding_lz}}

\address[tsinghua]{Department of Mathematical Sciences, Tsinghua University, Beijing, China}
\address[polyu]{Department of Applied Mathematics, The Hong Kong Polytechnic University, Hung Hom, Hong Kong}

\cortext[cor_author]{Corresponding Author. \emph{Email address:} lipingzhang@tsinghua.edu.cn}
\fntext[funding_ds]{Defeng Sun is partially supported by the Research Center for Intelligent Operations Research and RGC Senior Research Fellow Scheme No. SRFS2223-5S02.}
\fntext[funding_lz]{Liping Zhang is funded by the National Natural Sciences Foundation of China (Grant No. 12571323).}

\begin{abstract}
Rank-constrained matrix problems appear frequently across science and engineering. The convergence analysis of iterative algorithms developed for these problems often hinges on local error bounds, which correlate the distance to the feasible set with a measure of how much the constraints are violated. Foundational results in semi-algebraic geometry guarantee that such bounds exist, yet the associated exponents are generally not explicitly determined. This paper establishes a local H\"olderian error bound with an explicit exponent for the canonical rank-constrained affine feasibility set. This paper proves that, on any compact set, the distance to the feasible set is bounded by a power of a natural residual function capturing violations in both the rank and affine constraints. The exponent in this bound is given explicitly in terms of the problem's dimensions. This provides a fundamental quantitative result on the geometry of the solution set, paving the way for the convergence analysis of a broad class of numerical methods.
\end{abstract}


\end{frontmatter}

\section{Introduction}

The set of rank-constrained matrices satisfying a system of linear equations,
\begin{equation}
\cS = \{X \in \mathbb{R}^{m \times n} \mid \mathcal{A}(X) = b, \text{rank}(X) \le r\},
\label{eq:set_S}
\end{equation}
forms the feasibility region for a vast number of problems in modern data analysis and engineering. Applications include signal and image processing \cite{DavenportRomberg2016}, system identification and control \cite{FazelHindiBoyd2001}, and collaborative filtering, where it appears in the well-known matrix completion problem \cite{CandesRecht2012}. Despite its importance, the set \(\cS\) is notoriously challenging from a computational standpoint. The rank constraint introduces a non-convex, combinatorial structure, rendering the general feasibility problem NP-hard \cite{VandenbergheBoyd1996}.

A central question in the analysis of optimization algorithms is the rate of convergence to a solution set. The theory of error bounds provides a powerful framework for addressing this question by relating the distance from an arbitrary point to the solution set (i.e., \(\mathrm{dist}(X, \cS)\)) to the value of a residual function that measures constraint violation \cite{Pang1997,RockafellarWets1998}. For an iterative algorithm that generates a sequence \(\{X_k\}\), an error bound of the form \(\mathrm{dist}(X_k, \cS) \le \kappa \cdot \text{residual}(X_k)^\tau\) can provide a direct path to establishing convergence rates, showing that driving the residual to zero forces the iterates to approach the feasible set \(\cS\).

The study of error bounds for problems involving low-rank matrices has primarily proceeded along algorithmic lines. A significant body of work has analyzed convex relaxations, where the rank constraint is replaced by a constraint on the nuclear norm \cite{RechtFazelParrilo2010}. For solutions derived from such methods, error bounds have been established, often under statistical assumptions on the data or structural properties of the linear map \(\mathcal{A}\), such as the Restricted Isometry Property (RIP) \cite{RechtFazelParrilo2010}. These results, however, characterize the properties of solutions to an auxiliary problem rather than providing a direct geometric property of the original non-convex set~\eqref{eq:set_S}.

For non-convex methods that operate directly on low-rank factorizations (i.e., \(X = UV^T\)), local error bounds and related geometric properties like the quadratic growth condition or the Polyak-{\L}ojasiewicz (PL) inequality have been established. These results typically demonstrate that, near a solution, the objective landscape is well-behaved, which guarantees linear convergence for methods like gradient descent \cite{JainNetrapalliSanghavi2013,SunQuWright2018}. This line of inquiry provides crucial insight into algorithmic performance but couches the analysis in terms of a surrogate optimization landscape, not the geometry of the fundamental constraint set~\eqref{eq:set_S}.

Our work departs from this algorithmic focus to study the intrinsic geometry of the set~\eqref{eq:set_S} itself. Since \(\cS\) is a semi-algebraic set (defined by polynomial equalities and inequalities), the foundational work of {\L}ojasiewicz guarantees the existence of a H\"olderian error bound, often called a {\L}ojasiewicz inequality \cite{BochnakCosteRoy1998}. This classic result ensures that for a point \(X_0 \in \cS\) and a residual function \(f\), there exists a neighborhood of \(X_0\), a constant \(c\), and an exponent \(\tau \in (0, 1]\) such that \(\mathrm{dist}(X, \cS) \le c \cdot f(X)^\tau\) for all \(X\) in that neighborhood. While this establishes existence, the exponent \(\tau\) is generally non-constructive and can be arbitrarily close to zero, providing a weak guarantee.

The primary contribution of this paper is to establish an explicit, local H\"olderian error bound for the set~\eqref{eq:set_S}. We consider a natural residual function that combines the affine and rank constraint violations:
\[
    f(X) := \sum_{i=n-r+1}^{n} \sigma_i^2(X) + \frac{1}{2}\|\mathcal{A}(X) - b\|^2,
\]
where \(\sigma_i(X)\) denotes the \(i\)-th singular value of \(X\). Our main result, Proposition \ref{prop:EB}, shows that for any compact set \(\cK\), there exist constants \(c > 0\) and \(\tau > 0\), with \(\tau\) explicitly given in terms of the problem dimensions, such that
\[
    c \cdot \mathrm{dist}(X, \cS) \le f(X)^\tau \quad \text{for all } X \in \cK.
\]
This result provides a concrete, quantitative measure of the geometric regularity of the rank-constrained affine set. By providing an explicit exponent, our work moves beyond the abstract existence guarantees of classical semi-algebraic geometry and provides a foundational tool that can be used for the convergence analysis of any iterative method aiming to find a point in~\eqref{eq:set_S}.

\section{Notation and Preliminaries}

Throughout this paper, we let $\mathbb{R}^{m \times n}$ denote the space of $m \times n$ real matrices. This space is equipped with the trace inner product $\langle X, Y \rangle := \trace(X^T Y)$ and its induced norm, the Frobenius norm, denoted by $\|\cdot\|$. The singular values of a matrix $X \in \mathbb{R}^{m \times n}$ are denoted by $\sigma_1(X) \ge \sigma_2(X) \ge \dots \ge \sigma_{\min\{m,n\}}(X)$ in non-increasing order. The distance from a point $X$ to a set $A \subset \mathbb{R}^{m \times n}$ is defined as $\dist(X, A) := \inf_{Y \in A} \|X - Y\|$. The open ball centered at $X$ with radius $\epsilon$ is denoted by $\mathbb{B}(X, \epsilon)$, and its closure is $\bar{\mathbb{B}}(X, \epsilon)$. The open (closed) unit ball centered at the origin is denoted by $\mathbb{B}$ ($\bar{\mathbb{B}}$).
We use $I_k$ to denote the $k \times k$ identity matrix. For a linear map $\cA: \mathbb{R}^{m \times n} \to \mathbb{R}^l$, its adjoint is denoted by $\cA^*$.

Our analysis relies on tools from variational analysis \cite{RockafellarWets1998}. For a lower semicontinuous function $\phi: \mathbb{R}^l \to \mathbb{R}$, the Fr\'{e}chet subdifferential at $x \in \mathbb{R}^l$ is defined by
\[
    \hat{\partial} \phi(x) := \left\{ v \in \mathbb{R}^l \mid \liminf_{y \to x, y \ne x} \frac{\phi(y) - \phi(x) - \langle v, y-x \rangle}{\|y-x\|} \ge 0 \right\}.
\]
The limiting (or Mordukhovich) subdifferential, denoted by $\partial \phi(x)$, is defined as
\[
    \partial \phi(x) := \left\{ v \in \mathbb{R}^l \mid \exists x_k \overset{\phi}{\to} x, v_k \to v \text{ with } v_k \in \hat{\partial} \phi(x_k) \text{ for all } k \right\},
\]
where $x_k \overset{\phi}{\to} x$ means $x_k\to x$ and $\phi(x_k)\to \phi(x)$. If $\bar x\in \Re^l$ is a local minimum of $\phi$, then $0\in \hat{\partial}\phi(\bar x)$ by definition.
For our nonsmooth residual function $f(X)$, we define its minimal norm subgradient (or slope) as $m_f(X) := \dist(0, \partial f(X))$.

A key concept is that of semi-algebraic sets and functions. A set $A \subset \mathbb{R}^d$ is called semi-algebraic if it can be represented as a finite union of sets defined by a finite number of polynomial equalities and inequalities. A function is semi-algebraic if its graph is a semi-algebraic set. The semi-algebraic property is fundamental because it guarantees the existence of a local error bound via the {\L}ojasiewicz inequality. Our primary tool is an effective version of this inequality for polynomials, which provides an explicit exponent.

\begin{lemma}[{\cite[Theorem 4.2]{DAcuntoKurdyka2005}}]
\label{lem:poly_loja}
Let $P: \mathbb{R}^l \to \mathbb{R}$ be a polynomial of degree $\deg(P) \le d$. Let $\bar{x} \in \mathbb{R}^l$ be a critical point, i.e., $\nabla P(\bar{x}) = 0$. Then for any $r_0>0$ there exists $\epsilon>0$ and $c>0$ such that 
\[
    \|\nabla P(x)\| \ge c |P(x) - P(\bar{x})|^{1-\tau}
\]
for any $x\in \mathbb{B}(\bar x,r_0)$ with $|f(x)|<\epsilon$, where the {\L}ojasiewicz exponent $\tau$ is given by $\tau = R(l,d)^{-1}$, with $R(l, d) := d(3d-3)^{l-1}$. 
\end{lemma}
For a polynomial of degree $d=4$ in $l$ variables, this exponent is $\tau = (4 \cdot 9^{l-1})^{-1}$.

\section{Deriving the Error Bound}
Without loss of generality, assume that $\cS\neq\emptyset$, $\bar X\in \cS$ and $m\ge n$. The following analysis mainly follows the analysis in \cite{DinhPham2017}. Consider the function $g:\Re^{m\times n}\times \Re^{n\times (n-r)}$, defined by
\begin{equation*}
    g(X,V)=\langle XV,XV \rangle + \frac{1}{2}\|\cA(X)-b\|^2_F.
\end{equation*}
Then $g$ is a polynomial in $n(m+n-r)$ variables of degree $4$. Moreover, we have
\begin{equation*}
    \min_{V:V^{\top}V=I_{n-r}}g(X,V)=\sum_{i=n-r+1}^n \sigma^2_i(X)+\frac{1}{2}\|\cA(X)-b\|^2=f(X),
\end{equation*}
where $\sigma_1(X)\ge \sigma_2(X)\ge \dots\ge\sigma_n(X)$ are the singular values of $X$. Therefore, it is easy to know both $f$ and $g$ are semi-algebraic functions and $S$ is a semi-algebraic set. For each $X\in \Re^{m\times n}$, denote
\begin{equation*}
    E(X)=\{V\in \Re^{n\times (n-r)}\mid g(X,V)=f(X), \,V^{\top}V=I_{n-r}\},
\end{equation*}
which is a nonempty and compact set. For simplicity, denote $\cO_{n,r}=\{V\in \Re^{n\times (n-r)}\mid V^{\top}V=I_{n-r}\}$.
Then, we have the following lemma.
\begin{lemma}\label{lem:E_Holder_stable}
    The set-valued mapping $E:\Re^{m\times n}\rightrightarrows \Re^{n\times (n-r)}, X\mapsto E(X)$ is locally H\"older stable, i.e., for any fixed $\bar X\in \Re^{m\times n}$ and $\epsilon>0$, there exist $c,\alpha>0$ such that 
    \begin{equation*}
        E(X)\subset E(\bar X)+ c\|X-\bar X\|^{\alpha}  \bar{\mathbb{B}} \quad \text{ for all } X\in \mathbb{B}(\bar X,\epsilon).
    \end{equation*}
\end{lemma}
\begin{proof}
    We set the function $H(X,V):=|g(X,V)-f(X)|+\|V^{\top}V-I_{n-r}\|^2_F$. Then the function $H$ is semi-algebraic and locally Lipschitz. Additionally, we have $E(X)=\{V\mid H(X,V)=0\}$. Since $\cO_{n,r}$ is a compact set, it follows from the {\L}ojasiewicz inequality that there exist constants $c,\alpha>0$ such that 
    \begin{equation*}
        c\cdot\dist(V,E(\bar X))\le H(\bar X,V)^{\alpha} \quad \text{ for all } V\in \cO_{n,r}.
    \end{equation*}
    On the other hand, since $H$ is locally Lipschitz, it is therefore globally $L$-Lipschitz on the compact set $\bar{\mathbb{B}}(\bar X,\epsilon)\times \cO_{n,r}$ for some $L>0$. Thus, for $X\in \mathbb{B}(\bar X,\epsilon)$ and $V\in E(X)$, we have
    \begin{equation*}
        \begin{aligned}
            c\cdot\dist(V,E(\bar X))&\le H(\bar X,V)^{\alpha}=|H(\bar X,V)- H(X,V)|^{\alpha}\\
            &\le (L\|X-\bar X\|)^{\alpha}=L^{\alpha}\|X-\bar X\|^{\alpha}.
        \end{aligned}
    \end{equation*}
    This completes the proof.
\end{proof}
From Lemma \ref{lem:E_Holder_stable}, the following proposition is obvious.
\begin{proposition}\label{cor:Econtinuous}
    For each $\epsilon'>0$ and fixed $\bar X$, there exists a $\epsilon<\epsilon'$ such that for all $X\in \mathbb{B}(\bar X,\epsilon)$ and all $V\in E(X)$, we have
    \begin{equation*}
        \dist(V,E(\bar X))<\epsilon'.
    \end{equation*}
\end{proposition}

Next, we investigate the relationship between the generalized differentials of $g(X,V)$ and $f(X)$. Since $g$ is a continuously differentiable polynomial, it is natural to expect that certain structural properties of $g(X,V)$ can be transferred to $f(X)$.

\begin{lemma}
    \label{lem:gnl_diff}
    For all $X\in \Re^{m\times n}$ and $V\in E(X)$, the following statements hold:
    \begin{itemize}
        \item[\rm(i)] $\hat{\partial} f(X) \subset \{\nabla_X g(X,V)\}$. Moreover,
        \begin{equation*}
            \emptyset \neq {\partial} f(X) \subset \bigcup_{U \in E(X)} \{\nabla_X g(X,U)\} \quad \text{and} \quad
            m_{f}(X) \ge \inf_{U \in E(X)} \|\nabla_X g(X,U)\|.
        \end{equation*}
        \item[\rm(ii)] $\|\nabla_V g(X,V)\|\le 2f(X)$.
    \end{itemize}
\end{lemma}
\begin{proof}
    (i) Take arbitrary $Z\in \hat{\partial} f(X)$. By definition, for any $\epsilon>0$ there exists a $\delta>0$ such that 
    \begin{equation*}
        f(X+H)-f(X)-\langle H,Z\rangle \ge -\epsilon\|H\|\quad \text{ for all } H\in \mathbb{B}(0,\delta).
    \end{equation*}
    We define the function
    \begin{equation*}
        \phi(H):=g(X+H,V)-\langle Z,H\rangle +\epsilon\|H\|.
    \end{equation*}
    Then for all $H\in \mathbb{B}(0,\delta)$, we have
    \begin{equation*}
        \begin{aligned}
            \phi(H)&\ge f(X+H)-\langle Z,H\rangle +\epsilon\|H\|\\
            &\ge f(X)=g(X,V)=\phi(0).
        \end{aligned}
    \end{equation*}
    That is to say, $0$ is the local minima of $\phi$, therefore we have
    \begin{equation*}
        0\in \hat{\partial} \phi(0)\subseteq \nabla_X g(X,V)-Z+\epsilon\bar{\mathbb{B}}.
    \end{equation*}
    Thus, we have $\|\nabla_X g(X,V)-Z\|\le \epsilon$, this holds for any $\epsilon>0$. Then $Z=\nabla_X g(X,V)$, this holds for any $Z\in \hat{\partial} f(X)$, then $\hat{\partial} f(X)\in\{\nabla_X g(X,V)\}$.
    
    On the other hand, since $f$ is semi-algebraic, it is therefore continuously differentiable on a semi-algebraic dense open set $\cU\subset \Re^{m\times n}$ from the Cell Decomposition Theorem \cite[Theorem 4.2]{VanDenDriesMiller1996}. Then for all $X\in \cU$, we have
    \begin{equation*}
        \hat{\partial}f(X)=\partial f(X)=\{\nabla_X g(X,V)\}.
    \end{equation*}
    Since $\cU$ is dense, the function $g$ is $C^{\infty}$, the multi-function $E(X)$ is compact-valued and is locally H\"older stable (Lemma \ref{lem:E_Holder_stable}), we conclude that
    \begin{equation*}
        \emptyset \neq {\partial} f(X) \subset \bigcup_{U \in E(X)} \{\nabla_X g(X,U)\},
    \end{equation*}
    and the last inequality in (i) follows immediately.

    (ii) We have $g(X,V)=f(X)=\min_{U^{\top}U=I_{n-r}} g(X,U)$. Then there exists a multiplier $Y\in I_{n-r}$ such that $X^{\top}XV=VY$. Multiplying both sides of the equation on the left by $V^{\top}$, we obtain $Y=(XV)^{\top}XV$. Therefore, 
    \begin{equation}\label{eq:1}
        \begin{aligned}
            \|\nabla_V g(X,V)\|&=\|2X^{\top}XV\|=2\|VY\|\\
            &\le \|V\|_2\|Y\|_F=\|Y\|.
        \end{aligned}
    \end{equation}
    Additionally, we have 
    \begin{equation}\label{eq:2}
        2f(X)\ge 2\langle XV,XV\rangle =2\langle VY,V\rangle=2 \,\trace(Y). 
    \end{equation}
    Since $Y=(XV)^{\top}XV$ is a positive semidefinite matrix, we have 
    \begin{equation*}
        \|Y\|=\sqrt{\trace(Y^2)}\le \trace(Y).
    \end{equation*}
    Combined with \eqref{eq:1} and \eqref{eq:2}, the proof is completed.
\end{proof}
Having established a clear linkage between $f(X)$ and $g(X,V)$, our focus now shifts to deriving the error bound for $g$ since it is a polynomial.

\begin{proposition}\label{prop:exp_g}
    There exist constants $c>0$ and $\epsilon'>0$ such that
    \begin{equation*}
        \|\nabla g(X,V)\|\ge c\cdot g(X,V)^{1-\tau}
    \end{equation*}
    for all $X\in \mathbb{B}(\bar X,\epsilon')$ and $V\in \Re^{n\times (n-r)}$ with $\dist(V,E(\bar X))<\epsilon'$, where $$\tau:=\frac{1}{R(n(m+n-r),4)}:=\frac{1}{4\cdot 9^{n(m+n-r)-1}}.$$
\end{proposition}
\begin{proof}
    We know $f(\bar X)=0$, and $E(\bar X)$ is a compact set. For any $\bar V\in E(\bar X)$, we have $g(\bar X,\bar V)=0$ and $\nabla g(\bar X,\bar V)=0$ by simple observation. Recall that $g$ is a polynomial in $n(m+n-r)$ variables of degree $d=4$. From the {\L}ojasiewicz inequality (Lemma~\ref{lem:poly_loja}), we have constants $c>0,\epsilon'>0$ such that
    \begin{equation*}
        \|\nabla g(X,V)\|\ge c\cdot g(X,V)^{1-\tau}
    \end{equation*}
    for all $X\in \mathbb{B}(\bar X,\epsilon')$ and $V\in \mathbb{B}(\bar V,\epsilon')$. The conclusion then follows easily from the compactness of $E(\bar X)$.
\end{proof}

Now we are ready for the {\L}ojasiewicz inequality of the nonsmooth function $f$. Specifically, we have the following result. 

\begin{theorem}
    \label{thm:exp_f}
    There exist constants $c>0$ and $\epsilon>0$ such that 
    \begin{equation*}
        m_f(X)\ge c\cdot f(X)^{1-\tau}
    \end{equation*}
    for all $X\in \mathbb{B}(\bar X,\epsilon)$, where $\tau:=\frac{1}{R(n(m+n-r),4)}$.
\end{theorem}
\begin{proof}
    Let $c,\epsilon'$, and $\epsilon<\epsilon'$ be the positive constants such that Propositions \ref{cor:Econtinuous} and  \ref{prop:exp_g} hold. Take arbitrary $X\in \mathbb{B}(\bar X,\epsilon)$, we have $\dist(E(X),E(\bar X))\le \epsilon'$. Since $E(X)$ is compact, there exists a $V\in E(X)$ such that $$\|\nabla_X g(X,V)\|=\inf_{U \in E(X)} \|\nabla_X g(X,U)\|.$$ It follows from Lemma \ref{lem:gnl_diff} that
    \begin{equation*}
        m_f(X)\ge \|\nabla_X g(X,V)\| \quad \text{ and }\quad \|\nabla_V g(X,V)\|\le 2f(X).
    \end{equation*}
    Since all norms on finitely dimensional normed vector spaces are equivalent, we have  
    \begin{equation*}
        \begin{aligned}
            m_f(X)+2f(X)&\ge \|\nabla_X g(X,V)\|+\|\nabla_V g(X,V)\|\ \\
            &\ge c_1\|\nabla g(X,V)\| \\
            &\overset{(a)}{\ge} c_1c\cdot g(X,V)^{1-\tau}=c_1c \cdot f(X)^{1-\tau}
        \end{aligned}
    \end{equation*}
    for some $c_1$ only determined by $m,n,r$, and $(a)$ comes from $\dist(V,E(\bar X))<\epsilon'$ and Proposition \ref{prop:exp_g}. Note that $f(\bar X)=0$, we can diminish $\epsilon$ until we may assume 
    \begin{equation*}
        f(X)^{\tau}<\frac{c_1c}{4}, \quad \forall X\in \mathbb{B}(\bar X,\epsilon)
    \end{equation*}
    if necessary. Consequently, we obtain
    \begin{equation*}
        m_f(X)\ge (c_1c-2f(X)^{\tau})f(X)^{1-\tau}\ge \frac{c_1c}{2}f(X)^{1-\tau}, \quad\forall X\in \mathbb{B}(\bar X,\epsilon).
    \end{equation*}
    This completes the proof.
\end{proof}

The above {\L}ojasiewicz inequality leads to an error bound of $\cS$. 

\begin{proposition}
    \label{prop:EB}
    For any compact set $\cK\subset \Re^{m\times n}$, there exists a constant $c>0$ such that for all $X\in \cK$, the inequality
    \begin{equation*}
        c\cdot \dist(X,\cS)\le f(X)^{\tau}
    \end{equation*}
   holds with $\tau:=\frac{1}{R(n(m+n-r),4)}$.
\end{proposition}
\begin{proof}
    This can be derived by Theorem \ref{thm:exp_f}, \cite[Lemma 2.2]{DinhPham2017}, and the compactness of $\cK$.
\end{proof}

The potential looseness of the error bound in Proposition \ref{prop:EB} is a direct consequence of its reliance on the {\L}ojasiewicz exponent for polynomials, whose tightness still remains unknown. When the function $f$ satisfies certain regularity condition, a global bound can be established in the following. The proof of Theorem \ref{thm:globalEB} is omitted as it follows the same arguments used for \cite[Proposition 6.1]{DinhPham2017}.

\begin{assumption}\label{ass:regularity}
    For the function f, there exist constants $c>0$ and $R>0$ such that for all $\|X\|\ge R$, the inequality 
    \begin{equation*}
        m_f(X)\ge c
    \end{equation*} holds.
\end{assumption}

\begin{theorem}
    \label{thm:globalEB}
    If Assumption \ref{ass:regularity} holds, then the set $\cS$ is compact and there exists a constant $c>0$ such that for any $X\in \Re^{m\times n}$, the inequality
    \begin{equation*}
        c\cdot \dist (X,\cS)\le f(X)^{\tau} +f(X) 
    \end{equation*}
    holds with $\tau:=\frac{1}{R(n(m+n-r),4)}$.
\end{theorem}

\section{Conclusion}
In this paper, we constructed a concrete error bound for rank-constrained affine matrix sets. Our approach involved first defining a polynomial auxiliary function and establishing a key relationship between it and the problem's residual function. By applying the {\L}ojasiewicz inequality and its known exponent for polynomials, we then derived an explicit form for the error bound. We acknowledge that a limitation of this current result is its looseness. Therefore, future work will focus on identifying and verifying sufficient conditions under which a tighter, more practical error bound can be obtained. This remains a primary direction for our further research.


\end{document}